\documentclass[a4paper,12pt,notitlepage,oneside]{article}
\pdfoutput=1
\usepackage{amssymb,amsmath,amsthm} 
\usepackage[english]{babel} 
\usepackage[T1]{fontenc} 
\usepackage[utf8]{inputenc}        	

\makeatletter
\def\th@plain{%
  \thm@notefont{}
  \itshape 
}
\def\th@definition{%
  \thm@notefont{}
  \normalfont 
}
\makeatother

\usepackage{graphicx} 

\usepackage{enumerate}
\usepackage{enumitem}
\setlist{noitemsep}
\usepackage{tikz}
\usetikzlibrary{cd}
\usetikzlibrary{matrix}
\usepackage{mathtools}
\usepackage{xcolor}
\usepackage{url}

\theoremstyle{definition}
\newtheorem{definition}{Definition}[section] 
\numberwithin{definition}{section}
\newtheorem{example}[definition]{Example}
\newtheorem{remark}[definition]{Remark}

\theoremstyle{plain}
\newtheorem{lemma}[definition]{Lemma} 
\newtheorem{theorem}[definition]{Theorem}
\newtheorem{corollary}[definition]{Corollary}

\newcommand{\N}{\mathbb{N}}
\newcommand{\Z}{\mathbb{Z}}

\newcommand{\R}{\mathbb{R}}
\newcommand{\C}{\mathbb{C}}
\newcommand{\abs}[1]{\lvert #1 \rvert} 
\newcommand{\norm}[1]{\left\lVert #1 \right\rVert}
\newcommand{\powerset}[1]{\mathcal{P}(#1)}

\newcommand{\el}{l}

\DeclareMathOperator{\Exists}{\exists}
\DeclareMathOperator{\Forall}{\forall}
\DeclareMathOperator{\id}{\mathrm{id}}

\renewcommand{\phi}{\varphi}
\renewcommand{\epsilon}{\varepsilon}
\renewcommand{\iff}{\Leftrightarrow}

\title{A survey of Kolmogorov quotients}
\author{Teemu Pirttimäki\footnote{tealpi@utu.fi} \\[1cm] Department of Mathematics and Statistics \\ University of Turku}

\begin{document}

\maketitle

\begin{abstract}
Every topological space has a Kolmogorov quotient that is obtained by identifying topologically indistinguishable points, that is, points that are contained in exactly the same open sets. In this survey, we look at the relationship between topological spaces and their Kolmogorov quotients. In most natural examples of spaces, the Kolmogorov quotient is homeomorphic to the original space. A non-trivial relationship occurs, for example, in the case of pseudometric spaces, where the Kolmogorov quotient is a metric space. We also look at the topological indistinguishability relation in the context of topological groups and uniform spaces. The author is grateful to professor Tero Harju for his help and discussions on the topic.
\end{abstract}

\section{Introduction}

A \emph{Kolmogorov quotient} $X/{\equiv}$ is obtained from a topological space $X$ by identifying points $x$ and $y$ if and only if they have exactly the same open neighbourhoods. Such points are topologically indistinguishable; there is no sequence of operations on open sets that would give a set $A$ such that $x \in A$ and $y \not \in A$. Nothing topologically important to the space $X$ is lost in identifying these points.

The resulting space is a \emph{$T_0$-space}: a space where all points are topologically distinguishable. Most topological spaces of interest are $T_0$. A $T_0$-space is, arguably, aesthetically more pleasing than a space that is not $T_0$. In a $T_0$-space, every point serves a purpose. When studying the topology of $X$, there seems to be no reason to keep useless, superfluous points around.

The construction of the Kolmogorov quotient is simple, intuitive, and can be carried out for any topological space. If a mathematician comes across a space that is not naturally $T_0$, the unnecessary points can be left out from the space right at the beginning and the original space forgotten. Perhaps for this reason, the construction is not even mentioned in most textbooks on topology, and where it is mentioned, this is done very briefly, and proofs are generally omitted.

However, there are situations where it is inconvenient if a space is $T_0$. Such a situation occurs when one is interested in refinements of the topology: the more points there are in $X$, the more choices there are for refinements. The same is true for subspaces, though the loss here is not so dramatic: for each subspace $S \subseteq X$ that we lose, $X/{\equiv}$ retains a subspace homeomorphic to $S/{\equiv}$. Still, if one is interested in the specific points of the space, one might not wish to clump them together in equivalence classes.

Removing the $T_0$-property from a space can generate new properties for topological spaces. Given a property $P$ (for example, the Hausdorff separation axiom $T_2$) of a $T_0$-space we obtain a new property $P'$ by defining: a space $X$ has the property $P'$ if and only if $X/{\equiv}$ has the property $P$. Generally the arising property is interesting in itself and admits a more direct definition. In a similar vein, given a structure $S$ (for example, a metric) on a $T_0$-space we can define: a space $X$ has the structure $S'$ if and only if $X/{\equiv}$ has the property $S$.

This survey is not about $T_0$-spaces, but focuses rather on the relationship between spaces and their Kolmogorov quotients. It appears that no comprehensive treatment on the matter has been published, and as stated before, standard textbooks often omit the construction entirely. As our sources don't usually give proofs, it seems unnecessary to cite each theorem individually. Various results presented here can be found without proofs in \cite{Caicedo}--\cite{HAF}. The notes in \cite{Clark} contain some proofs. We add here many details missing from these sources and present the results in a more general form when possible.

The author's interest in the subject was sparked by study of abstract model theory, specifically the paper \cite{Caicedo} by X. Caicedo, where Kolmogorov quotients are used in a topological proof of Lindström's theorem.

Kolmogorov spaces apparently get their name from an unpublished manu\-script by Andrey Kolmogorov (see \cite{Hofmann}, p.~1).

\section{Preliminaries}

In this section, we present notions that are required to understand the main results, but not directly connected to Kolmogorov quotients.

Given a topological space $X$ and a subset $A \subseteq X$, we write $A^c$ for the complement $X \setminus A$ and $\overline{A}$ for the closure of $A$. The powerset of a set $S$ is denoted by $\powerset S$.

\begin{definition}
Let $(X, \tau)$ be a topological space. The \emph{Borel algebra} of $X$ is a collection $\Sigma_X$ of subsets of $X$, defined inductively:
\begin{enumerate}
\item $\tau \subseteq \Sigma_X$.
\item If $\Delta \subseteq \Sigma_X$ is countable, then $\bigcup_{B \in \Delta} B \in \Sigma_X$.
\item If $B \in \Sigma_X$, then $B^c \in \Sigma_X$.
\end{enumerate}
The sets contained in $\Sigma_X$ are called the \emph{Borel sets} of $X$.
\end{definition}

By De Morgan's laws, intersections of countably many Borel sets are Borel sets. Hence also the relative complements $B_1 \setminus B_2 = B_1 \cap B^c_2$ are Borel sets.

\begin{example}
Let $a$ be an arbitrary real number. By rule 1, $\R$ is a Borel set, and the open intervals $(-\infty, a)$ and $(a, \infty)$ are Borel sets. By rule 3, $(-\infty, a] = \R \setminus (a, \infty)$ and $[a, \infty) = \R \setminus (-\infty, a)$  are Borel sets. Their intersection $\{a\} = (-\infty, a] \cap [a, \infty)$ is a Borel set. Since $a$ was arbitrary, all countable sets of reals are Borel sets by rule 2.
\end{example}

Since the Borel algebra is closed with respect to finite unions and intersections, and contains the complements $X \setminus B$ for all $B \in \Sigma_X$, it is a Boolean algebra.

\begin{definition}
Let $X$ be a set. A collection $\mathcal{F}$ of subsets of $X$ is a \emph{filter} on $X$ if
\begin{enumerate}
\item[F1.] $X \in \mathcal{F}$;
\item[F2.] $A \in \mathcal{F}$ and $A \subseteq B \subseteq X$ implies $B \in \mathcal{F}$;
\item[F3.] $A, B \in \mathcal{F}$ implies $A \cap B \in \mathcal F$.
\end{enumerate}
A filter $\mathcal F$ is \emph{proper} if $\mathcal{F} \neq \powerset X$. If there is no proper filter on $X$ that includes $\mathcal{F}$ as a proper subset, we say that $\mathcal{F}$ is an \emph{ultrafilter}.
\end{definition}

\begin{example}
Let $X$ be a topological space, and let $x \in X$. A set $S \subseteq X$ is a \emph{neighbourhood} of the point $x$ if $x \in U \subseteq S$ for some open set $U$. For each $x \in X$, denote by $\mathcal{N}(x)$ the collection of neighbourhoods of $x$. Then $\mathcal{N}(x)$ is a filter: First, $X$ is a neighbourhood of $x$. Second, if $x \in U \subseteq A$, where $U$ is open, then $B$ is a neighbourhood of $x$ for all $B \subseteq X$ such that $A \subseteq B$. Third, if $x \in U \subseteq A$ and $x \in V \subseteq B$, where $U$ and $V$ are open, then $U \cap V$ is open and $x \in U \cap V \subseteq A \cap B$, so $A \cap B$ is a neighbourhood of $x$.
\end{example}

\begin{example}
Let $X$ be a set and $x \in X$. The set $\mathcal{U}_x = \{S \in \powerset X \mid x \in S\}$ is an ultrafilter: First, $x \in X$. Second, if $x \in A$ and $A \subseteq B$, then $x \in B$. Third, if $x \in A$ and $x \in B$, then $x \in A \cap B$. We have shown that $\mathcal{U}_x$ is a filter. To show that it is an ultrafilter, suppose that $\mathcal{U}_x$ is a proper subset of a filter $\mathcal{F}$. Then $\mathcal{F}$ contains some set $S \subseteq X$ such that $x \not \in S$. By F3, $\emptyset = S \cap \{x\} \in \mathcal{F}$. Since $\emptyset \subseteq B$ for all $B \subseteq X$, we have $\mathcal F = \powerset X$ by F2.
\end{example}

\begin{definition}
Let $X$ be a topological space and $\mathcal{F}$ a filter on $X$. If $\mathcal{N}(x) \subseteq \mathcal{F}$ for some $x \in X$, then we say that $\mathcal{F}$ \emph{converges to} $x$. A point $x \in X$ is a \emph{cluster point} of $\mathcal{F}$ if there exists a proper filter that includes $\mathcal{F} \cup \mathcal{N}(x)$ as a subset.
\end{definition}

In every topological space, $\mathcal{N}(x)$ converges to $x$ and has $x$ as a cluster point. The same is true for $\mathcal{U}_x$, since $\mathcal N(x) \subseteq \mathcal U_x$.

\begin{definition}
A \emph{directed set} is a pair $(D, \geq)$, where $D$ is a set and $\geq$ a partial order on $D$ such that for all $i,j \in D$ there exists an element $k \in D$ for which $k \geq i$ and $k \geq j$. A \emph{net} in a topological space $X$ is a map from a directed set $D$ to $X$. A net $f$ \emph{converges to} $x \in X$ if for every $U \in \mathcal{N}(x)$, there exists an element $i_0 \in D$ such that $f(i) \in U$ for all $i \geq i_0$.
\end{definition}

\begin{remark} \label{net remark}
Every sequence is a net where $D = \N$ and $\geq$ is the usual order relation on $\N$. Filter and net convergence are equivalent (\cite{HAF}, pp. 158--160). For this reason, we will restrict ourselves to using only filters.
\end{remark}

\section{Kolmogorov quotients}

Let $X$ be a topological space. We define an equivalence relation ${\equiv} \subseteq X^2$ by letting $x \equiv y$ if and only if every open neighbourhood of $x$ is an open neighbourhood of $y$ and vice versa. If $x \equiv y$, we say that the points $x$ and $y$ are \emph{topologically indistinguishable}. Otherwise $x$ and $y$ are \emph{topologically distinguishable}, and we write $x \not \equiv y$. A space where all pairs of distinct points are topologically distinguishable is called a \emph{$T_0$-space} or a \emph{Kolmogorov space}. Most spaces studied by mathematicians are $T_0$.

\begin{example}
A space with the trivial topology is not $T_0$, unless it has less than two points.
\end{example}

\begin{example}
All Hausdorff spaces are $T_0$. This includes all discrete spaces and the space $\R$ with the euclidean topology.
\end{example}

\begin{example}
Let $X = \{0,1\}$ and $\tau = \{\emptyset, \{1\}, \{0,1\}\}$. The \emph{Sierpiński space} $(X, \tau)$ is $T_0$ but not Hausdorff.
\end{example}

\begin{example}
The product of $\R$ with the euclidean topology and $\R$ with the trivial topology is not $T_0$: indeed, the points $(1,0)$ and $(1,1)$ are topologically indistinguishable.
\end{example}

We will see more examples later. In the meanwhile, the following lemma should provide intuition into topological indistinguishability.

\begin{lemma} \label{indistinguishability equivalences}
Let $X$ be a topological space and $x, y \in X$. The following statements are equivalent:
\begin{enumerate}[label = \rm{(\roman*)}]
\item $x \equiv y$;
\item $\mathcal N(x) = \mathcal N(y)$;
\item $x$ and $y$ are contained in the same basic open sets;
\item $x$ and $y$ are contained in the same subbasic open sets;
\item $x$ and $y$ are contained in the same open sets;
\item $x$ and $y$ are contained in the same closed sets;
\item $\overline{\{x\}} = \overline{\{y\}}$;
\item $x$ and $y$ are contained in the same Borel sets;
\item a filter or net that converges to $x$ converges also to $y$, and vice versa;
\item a filter or net that has $x$ as a cluster point has also $y$ as a cluster point, and vice versa.
\end{enumerate}
\end{lemma}
\begin{proof}
We prove the equivalences $\mathrm{(ix)} \iff \mathrm{(ii)} \iff \mathrm{(i)} \iff \mathrm{(v)} \iff \mathrm{(iii)} \iff \mathrm{(iv)}$, $\mathrm{(v)} \iff \mathrm{(vi)} \iff \mathrm{(vii)}$, $\mathrm{(v)} \iff \mathrm{(viii)}$ and $\mathrm{(ii)} \Rightarrow \mathrm{(x)} \Rightarrow\mathrm{(v)}$.

Suppose $\mathrm{(i)}$. Suppose that $S$ is a neighbourhood of $x$. Then there is an open set $U$ such that $x \in U \subseteq S$. By $\mathrm{(i)}$, $y \in U \subseteq S$, so $S$ is a neighbourhood of $y$. By symmetry, every neighbourhood of $y$ is a neighbourhood of $x$. This proves that $\mathrm{(i) \Rightarrow \mathrm{(ii)}}$. The other direction $\mathrm{(ii)} \Rightarrow \mathrm{(i)}$ is trivial.

Trivially $\mathrm{(i)} \Leftrightarrow \mathrm{(v)} \Rightarrow \mathrm{(iii)}$. Suppose $\mathrm{(iii)}$. Let $\mathcal B$ be a basis and $B$ an open set. Then $B = \bigcup_{B' \in S} B'$ for some subset $S \subseteq \mathcal{B}$, and for each $B'$, $x \in B'$ if and only if $y \in B'$. Then
\begin{align*}
x \in B & \iff x \in B' \text{ for some } B' \in S \\
		& \iff y \in B' \text{ for some } B' \in S \\
		& \iff y \in B,
\end{align*}
so $\mathrm{(v)}$ holds, which proves that $\mathrm{(iii)} \Rightarrow \mathrm{(v)}$.

Every subbasic set is in the basis determined by the subbasis. Thus $\mathrm{(iii)} \Rightarrow \mathrm{(iv)}$. Suppose $\mathrm{(iv)}$. Let $\mathcal{S}$ be a subbasis and $B$ a basic open set in the basis determined by $\mathcal{S}$. Then $B = \bigcap_{B' \in S} B'$ for some finite subset $S \subseteq \mathcal{S}$, and for each $B'$, $x \in B'$ if and only if $y \in B'$. Then
\begin{align*}
x \in B & \iff x \in B' \text{ for all } B' \in S \\
		& \iff y \in B' \text{ for all } B' \in S \\
		& \iff y \in B,
\end{align*}
so $\mathrm{(iii)}$ holds, which proves that $\mathrm{(iv)} \Rightarrow \mathrm{(iii)}$.

Claim $\mathrm{(v)}$ states that for all open sets $U$, $x \in U$ if and only if $y \in U$. Claim $\mathrm{(vi)}$ states that for all open sets $U$, $x \in U^c$ if and only if $y \in U^c$. It is then clear that $\mathrm{(v)} \Leftrightarrow \mathrm{(vi)}$.

Suppose $\mathrm{(vi)}$ holds. Then
\[
\overline{\{x\}} = \bigcap_{\substack{x \in F \\ F \text{ closed}}} F = \bigcap_{\substack{y \in F \\ F \text{ closed}}} F = \overline{\{y\}}.
\]
Hence $\mathrm{(vi)} \Rightarrow \mathrm{(vii)}$. Suppose then that $\mathrm{(vi)}$ does not hold. Without loss of generality, we may assume that there is a closed set containing $x$ that does not contain $y$. Then $y \not \in \overline{\{x\}}$, even though $y \in \overline{\{y\}}$. Thus $\mathrm{(vii)}$ fails. By contraposition, $\mathrm{(vii)} \Rightarrow \mathrm{(vi)}$.

Since open sets are Borel sets, $\mathrm{(viii)} \Rightarrow \mathrm{(v)}$. Suppose then that $\mathrm{(v)}$ holds. We use induction to show that $\mathrm{(viii)}$ follows:

Suppose $B$ is an open subset of $X$. Then $x \in B$ if and only if $y \in B$.

Suppose $B = \bigcup_{B' \in \Delta} B'$, where $\Delta \subseteq \Sigma_{X}$ is countable, and for each $B'$, $x \in B'$ if and only if $y \in B'$. Then
\begin{align*}
x \in B & \iff x \in B' \text{ for some } B' \in \Delta \\
		& \iff y \in B' \text{ for some } B' \in \Delta \\
		& \iff y \in B.
\end{align*}

Suppose $B = B_1^c$, where $B_1 \in \Sigma_X$ and $x \in B_1$ if and only if $y \in B_1$. Then
\[
x \in B \iff x \not \in B_1 \iff y \not \in B_1 \iff y \in B.
\]

This concludes the proof that $\mathrm{(v)} \Rightarrow \mathrm{(viii)}$.

For $\mathrm{(ix)}$ and $\mathrm{(x)}$, we consider filters only; see remark \ref{net remark}.

Suppose $\mathrm{(ii)}$. Let $\mathcal{F}$ be a filter on $X$. Then $\mathcal{N}(x) \subseteq \mathcal{F}$ if and only if $\mathcal{N}(y) \subseteq \mathcal{F}$, so if $\mathcal{F}$ converges to one of the points $x$ and $y$, it converges to both of them. This proves $\mathrm{(ii)} \Rightarrow \mathrm{(ix)}$. For the other direction, suppose $\mathrm{(ii)}$ does not hold. Then one of the sets $\mathcal{N}(x)$ and $\mathcal{N}(y)$ is not a subset of the other; without loss of generality, we may assume $\mathcal{N}(x) \not \subseteq \mathcal{N}(y)$. Then $\mathcal{N}(y)$ is a filter that includes $\mathcal{N}(y)$ but not $\mathcal{N}(x)$. Therefore, $\mathcal{N}(y)$ converges to $y$ but not to $x$. Thus $\mathrm{(ix)}$ fails, so by contraposition, $\mathrm{(ix)} \Rightarrow \mathrm{(ii)}$.

Trivially $\mathrm{(ii)} \Rightarrow \mathrm{(x)}$. We show that $\mathrm{(x)} \Rightarrow{\mathrm{(v)}}$. Suppose that $\mathrm{(x)}$ holds. Let $\mathcal{U}$ be an ultrafilter. By $\mathrm{(x)}$, $\mathcal{U} \cup \mathcal{N}(x)$ is included in a proper filter if and only if $\mathcal{U} \cup \mathcal{N}(y)$ is included in a proper filter. Since ultrafilters are maximal proper filters, the only proper filter that could include these is $\mathcal{U}$. Hence
\begin{equation} \label{ultrafilters converge to same points}
\begin{array}{ccc}
\mathcal{N}(x) \subseteq \mathcal{U} & \text{iff} & \mathcal{N}(y) \subseteq \mathcal{U}.
\end{array}
\end{equation}
Suppose $\mathrm{(v)}$ fails. Without loss of generality, we may assume that there is an open set $V$ such that $x \in V$ and $y \not \in V$. Let $\mathcal{U}_y$ be the ultrafilter $\{S \in \powerset X \mid y \in S\}$. Clearly $\mathcal{N}(y) \subseteq \mathcal{U}_y$. On the other hand, $V \not \in \mathcal{U}_y$, but $V \in \mathcal{N}(x)$. Therefore, $\mathcal{N}(x) \not \subseteq \mathcal{U}_y$, which contradicts (\ref{ultrafilters converge to same points}).
\end{proof}

\begin{remark}
The formulation of $\mathrm{(iii)}$ and $\mathrm{(iv)}$ in the above lemma is deliberately a bit vague. To show $\mathrm{(iv)} \Rightarrow \mathrm{(iii)} \Rightarrow \mathrm{(v)}$, it is enough to suppose these properties for \emph{some} basis or subbasis that generates the topology of $X$. Then $\mathrm{(v)}$ implies $\mathrm{(iii)}$ and $\mathrm{(iv)}$ for \emph{all} bases and subbases for the same topology.
\end{remark}

\begin{example}
Let $U_m = \{n \in \N \mid m \text{ divides } n\}$ for all $m \in \Z_+$. Then $\mathcal S = \{\N\} \cup \{U_p \mid p \text{ is a prime}\}$ is a subbasis of a topology on $\N$. By lemma \ref{indistinguishability equivalences} $\mathrm{(iv)}$, $x \equiv y$ if and only if $x$ and $y$ have the same prime factors.
\end{example}

Given a topological space $X$, we denote by $\eta(x)$ the equivalence class of $x \in X$ with respect to $\equiv$, that is, $\eta(x) = \{y \in X \mid y \equiv x\}$. The following theorem gives a simple formula for the equivalence classes.

\begin{theorem} \label{equivalence class is intersection of closure and open neighbourhoods}
Let $(X, \tau)$ be a topological space. For all $x \in X$,
\[
\eta(x) = \overline{\{x\}} \cap \bigcap_{\substack{U \in \tau \\ x \in U}} U.
\]
\end{theorem}
\begin{proof}
Let $y \in \eta(x)$. Then by lemma \ref{indistinguishability equivalences} \rm{(v)} and \rm{(vi)}, $x$ and $y$ are contained in the same open sets and same closed sets. Hence
\[
y \in \bigcap_{\substack{F^c \in \tau \\ x \in F}} F \cap \bigcap_{\substack{U \in \tau \\ x \in U}} U  = \overline{\{x\}} \cap \bigcap_{\substack{U \in \tau \\ x \in U}} U,
\]
which proves that
\[
\eta(x) \subseteq \overline{\{x\}} \cap \bigcap_{\substack{U \in \tau \\ x \in U}} U.
\]

We prove inclusion in the other direction by contraposition. Suppose $y \not \in \eta(x)$. Then by lemma \ref{indistinguishability equivalences} \rm{(vii)}, there are two possibilities: $y \not \in \overline{\{x\}}$, or $x \not \in \overline{\{y\}}$. In the first case,
\[
y \not \in \overline{\{x\}} \cap \bigcap_{\substack{U \in \tau \\ x \in U}} U.
\]
In the latter case, $\overline{\{y\}}^c$ is an open set containing $x$ and not containing $y$, so
\[
y \not \in \overline{\{x\}} \cap \bigcap_{\substack{U \in \tau \\ x \in U}} U.
\]
By contraposition, we have
\[
\overline{\{x\}} \cap \bigcap_{\substack{U \in \tau \\ x \in U}} U \subseteq \eta(x),
\]
and the claim follows.
\end{proof}

\begin{remark}
We could also write the equivalence class as the intersection of all Borel sets containing $x$:
\[
\eta(x) = \smashoperator[r]{\bigcap_{\substack{B \in \Sigma_X \\ x \in B}}} \enspace B.
\]
Using lemma \ref{indistinguishability equivalences} \rm{(viii)}, we can write a proof for this similar to that of theorem \ref{equivalence class is intersection of closure and open neighbourhoods}.
\end{remark}

\begin{corollary}
For all $x \in X$,
\[
\eta(x) \subseteq \smashoperator[r]{\bigcap_{U \in \mathcal N(x)}} \enspace U.
\]
\end{corollary}
\begin{proof}
By theorem \ref{equivalence class is intersection of closure and open neighbourhoods},
\[
\eta(x) \subseteq \smashoperator[r]{\bigcap_{\substack{U \text{ open} \\ x \in U}}} \enspace U = \smashoperator[r]{\bigcap_{U \in \mathcal N(x)}} \enspace U.
\]
The last equality follows from the fact that every open neighbourhood $U$ of $x$ is the intersection of those neighbourhoods of $x$ that include $U$.
\end{proof}

\begin{corollary} \label{equivalence class subset of closure}
For all $x \in X$, $\eta(x) \subseteq \overline{\{x\}}$.
\end{corollary}
\begin{proof}
This follows from theorem \ref{equivalence class is intersection of closure and open neighbourhoods}.

An alternative proof: Let $y \in \eta(x)$. Then $\eta(y) = \eta(x)$. By lemma \ref{indistinguishability equivalences} \rm{(vii)}, $\overline{\{y\}} = \overline{\{x\}}$. Since $y \in \overline{\{y\}}$, we have $y \in \overline{\{x\}}$.
\end{proof}

Given a topological space $X$, we define $X/{\equiv}$ as the topological space, where the space as a set is the set of equivalence classes under $\equiv$, and the topology is the finest such topology that the quotient map $\eta \colon X \to X/{\equiv}$ that maps each element $x \in X$ to its equivalence class $\eta(x)$ is continuous. In other words, the open sets of $X/{\equiv}$ are precisely those sets whose preimage under $\eta$ is open in $X$. We call the space $X/{\equiv}$ the \emph{Kolmogorov quotient} of $X$.

The Kolmogorov quotient is always a Kolmogorov space, but a rigorous proof of this will have to wait until theorem \ref{quotient is T_0}.

The continuity of $\eta$ already lets us know some things about the relationship between $X$ and $X/{\equiv}$; for example, if $A \subseteq X$ is compact, then so is $\eta(A)$.

\begin{example}
Take the set $X = \{1,2,3,4\}$ with the clopen basis $\{\{1,2\}, \\ \{3,4\}\}$. The Kolmogorov quotient $X/{\equiv}$ is the two-element set $\{\eta(1), \eta(3)\} = \{\{1,2\}, \{3,4\}\}$ with the discrete topology.
\end{example}

\begin{example}
The Kolmogorov quotient of any nonempty set with the trivial topology is a space consisting of a single point.
\end{example}

\begin{example} \label{L^p}
Let $p \geq 1$. Let $L^p$ be the set of all measurable functions $f$ from a measure space $(S, \Sigma, \mu)$ to $\R$ such that
\[
\int_S {\abs f}^p \, \mathrm{d}\mu < \infty.
\]
Defining $(f + g)(x) = f(x) + g(x)$ and $(\lambda f)(x) = \lambda f(x)$ for all $x \in S$, $f, g \in L^p$ and $\lambda \in \R$ makes $L^p$ a vector space.
Denote
\[
\norm{f}_p = \left(\int_S {\abs f}^p \, \mathrm{d}\mu  \right)^\frac{1}{p}.
\]
The map $f \mapsto \norm{f}_p$ is a \emph{seminorm} (see definition \ref{seminorm}). It is not a norm, since $\norm{f}_p = 0$ for all functions $f$ that agree with the zero function almost everywhere. In the Kolmogorov quotient $\mathcal{L}^p = L^p/{\equiv}$, on the other hand, this seminorm becomes a norm. In section \ref{Pseudometrics and seminorms}, we show that this happens to all seminorms when we take the Kolmogorov quotient. The spaces $\mathcal{L}^p$ are important in analysis and measure theory (\cite{Lp-spaces}).
\end{example}

\begin{example}
A discrete version of example \ref{L^p} is obtained by taking the measure space $\N$ with the counting measure, i.e. the measure of a subset of $\N$ is its cardinality. In this case, the space consists of sequences where
\[
\sum_{n=0}^\infty {\abs{x_n}^p} < \infty,
\]
and
\[
\norm{(x_n)}_p = \left(\sum_{n=0}^\infty {\abs{x_n}^p} \right)^\frac{1}{p}.
\]
\end{example}

Based on the quotient map $\eta \colon X \to X/{\equiv}$, we define two maps $\eta^\rightarrow \colon \Sigma_X \to \Sigma_{X/{\equiv}}$ and $\eta^\leftarrow \colon \Sigma_{X/{\equiv}} \to \powerset X$ as follows:
\[
\eta^\rightarrow(B) = \eta(B) = \{\eta(x) \mid x \in B\},
\]
and
\[
\eta^\leftarrow(B') = \eta^{-1}(B') = \{ x \in X \mid \eta(x) \in B'\}
\]
for all $B \in \Sigma_X$ and $B' \in \Sigma_{X/{\equiv}}$.

\begin{theorem}
The map $\eta^\rightarrow$ is an isomorphism between the Boolean algebras $\Sigma_X$ and $\Sigma_{X/{\equiv}}$.
\end{theorem}
\begin{proof}
We wish to prove that $\eta^\rightarrow$ is bijective and preserves binary unions, binary intersections and complements.

For injectivity, let $B_1$, $B_2 \in \Sigma_X$. If $\eta^\rightarrow(B_1) \subseteq \eta^\rightarrow(B_2)$, then for all $x_1 \in B_1$, there is a point $x_2 \in B_2$ such that $\eta(x_1) = \eta(x_2)$. Then $x_1 \equiv x_2$, and they are in the same Borel sets; in particular $x_1 \in B_2$. Hence $B_1 \subseteq B_2$. Similarly, if $\eta^\rightarrow(B_2) \subseteq \eta^\rightarrow(B_1)$, then $B_2 \subseteq B_1$. Consequently, $\eta^\rightarrow(B_1) = \eta^\rightarrow(B_2)$ implies $B_1 = B_2$.

Surjectivity can be proved by induction:

Suppose $B$ is an open subset of $X/{\equiv}$. Then $\eta^\leftarrow(B)$ is open by definition of $X/{\equiv}$, and hence $\eta^\leftarrow(B) \in \Sigma_X$. Then $B = \eta^\rightarrow(\eta^\leftarrow(B))$.

Suppose $B = \bigcup_{B' \in \Delta} B'$, where $\Delta \subseteq \Sigma_{X/{\equiv}}$ is countable and each $B'$ is an image of some $\eta^\leftarrow(B') \in \Sigma_X$. Then $\eta^\leftarrow(B) = \eta^\leftarrow\left(\bigcup_{B' \in \Delta} B' \right) = \bigcup_{B' \in \Delta} \eta^\leftarrow(B') \in \Sigma_X$, and $B = \eta^\rightarrow(\eta^\leftarrow(B))$. 

Suppose $B = B_1^c$, where $B_1 \in \Sigma_{X/{\equiv}}$ and $B_1$ is the image of some set $\eta^\leftarrow(B_1) \in \Sigma_X$. Then $\eta^\leftarrow(B_1^c) = [\eta^\leftarrow(B_1)]^c \in \Sigma_X$ and $B_1^c = \eta^\rightarrow(\eta^\leftarrow(B_1^c))$.


Thus $\eta^\rightarrow$ is surjective.

Let $B_1$, $B_2 \in \Sigma_X$. Then
\begin{align*}
t \in \eta^\rightarrow(B_1 \cup B_2) & \Leftrightarrow \text{there is } x \in B_1 \cup B_2 \text{ such that } t = \eta(x)\\
						 & \Leftrightarrow \text{there is } x \in B_1 \text{ or } x \in B_2 \text{ such that } t = \eta(x) \\
						 & \Leftrightarrow t \in \eta^\rightarrow(B_1) \text{ or } t \in \eta^\rightarrow(B_2) \\
						 & \Leftrightarrow t \in \eta^\rightarrow(B_1) \cup \eta^\rightarrow(B_2),
\end{align*}
so $\eta^\rightarrow$ preserves binary unions.

For complements, we note first that the images $\eta^\rightarrow(B)$ and $\eta^\rightarrow(B^c)$ are disjoint for all Borel sets $B \in \Sigma_X$. Indeed, if there were points $x_1 \in B$, $x_2 \in B^c$ such that $\eta(x_1) = \eta(x_2)$, then $x_1$ and $x_2$ would be contained in the same Borel sets. Hence $x_1 \in B$ and $x_1 \in B^c$, which is a contradiction.

Let $B \in \Sigma_X$. Since $\eta^\rightarrow$ preserves binary unions, we have
\[
X/{\equiv} = \eta^\rightarrow(X) = \eta^\rightarrow(B \cup B^c) = \eta^\rightarrow(B) \cup \eta^\rightarrow(B^c).
\]
Since $\eta^\rightarrow(B)$ and $\eta^\rightarrow(B^c)$ are disjoint, we must have $\eta^\rightarrow(B^c) = [\eta(B)]^c$.

Finally, the fact that $\eta^\rightarrow$ preserves binary intersections is now easily seen from De Morgan's laws:
\begin{align*}
\eta^\rightarrow(B_1 \cap B_2) 	& = \eta^\rightarrow((B_1^c \cup B_2^c)^c) \\
								& = \{[\eta^\rightarrow(B_1)]^c \cup [\eta^\rightarrow(B_2)]^c\}^c \\
								& = \eta^\rightarrow(B_1) \cap \eta^\rightarrow(B_2).
\end{align*}
\end{proof}

\begin{corollary} \label{eta is open}
The quotient map $\eta$ is open, i.e. if $A \subseteq X$ is open, then $\eta^\rightarrow(A)$ is open.
\end{corollary}
\begin{proof}
By injectivity of $\eta^\rightarrow$, we have $A = \eta^\leftarrow(\eta^\rightarrow(A))$. By the definition of the Kolmogorov quotient, $\eta^\rightarrow(A)$ is open.
\end{proof}

\begin{corollary} \label{eta is closed}
The quotient map $\eta$ is closed, i.e. if $A \subseteq X$ is closed, then $\eta^\rightarrow(A)$ is closed.
\end{corollary}
\begin{proof}
Since $\eta^\rightarrow$ preserves complements,
\begin{align*}
A \text{ is closed}	& \Leftrightarrow A^c \text{ is open} \\
					& \Rightarrow \eta^\rightarrow(A^c) \text{ is open} \\
					& \Leftrightarrow [\eta^\rightarrow(A)]^c \text{ is open} \\
					& \Leftrightarrow \eta^\rightarrow(A) \text{ is closed.}
\end{align*}
\end{proof}

The converses of corollaries \ref{eta is open} and \ref{eta is closed} generally do not hold: there can be a set $A \subseteq X$ such that even if $\eta(A)$ is open/closed, $A$ is not open/closed. Take the set $X = \{1,2,3,4\}$ with the clopen basis $\{\{1,2\}, \{3,4\}\}$. Let $A = \{1,3\}$. Clearly $A$ is neither open nor closed. On the other hand, $\eta(A) = \{\eta(1), \eta(3)\} = X/{\equiv}$, which is clopen.

Recall that a space is Kolmogorov or $T_0$ if every pair of points is topologically distinguishable. The following two results show that the name of Kolmogorov quotients is not arbitrarily chosen.

\begin{theorem}\label{quotient is T_0}
Let $X$ be a topological space. The Kolmogorov quotient of $X$ is a Kolmogorov space.
\end{theorem}
\begin{proof}
Suppose to the contrary that there are $x_1, x_2 \in X$ such that $\eta(x_1)$ and $\eta(x_2)$ are distinct topologically indistinguishable points in $X/{\equiv}$. Then $x_1$ and $x_2$ are topologically indistinguishable; otherwise $\eta(x_1) = \eta(x_2)$. Without loss of generality, we may assume that there is an open set $U$ such that $x_1 \in U$ and $x_2 \not \in U$. Now $\eta^\rightarrow(U)$ is open in $X/{\equiv}$ by corollary \ref{eta is open}, and $\eta(x_1) \in \eta^\rightarrow(U)$. Since $\eta(x_1)$ and $\eta(x_2)$ are topologically indistinguishable, $\eta(x_2) \in \eta^\rightarrow(U)$. But then $x_2 \in \eta^{-1}(\eta(x_2)) = U$; a contradiction.
\end{proof}

\begin{corollary}
A space is Kolmogorov if and only if it is homeomorphic to the Kolmogorov quotient of itself.
\end{corollary}
\begin{proof}
Suppose $X$ is a Kolmogorov space. Then every equivalence class consists of exactly one element, so $\eta$ is an open continuous bijection, that is, a homeomorphism.

Suppose then that there is a homeomorphism $f \colon X \to X/{\equiv}$. In particular, $f$ is injective, so every equivalence class consists of exactly one element, that is, $X$ is Kolmogorov.
\end{proof}

\begin{lemma} \label{continuous functions preserve equivalence classes}
Let $X$ and $Y$ be topological spaces and $f \colon X \to Y$ continuous. If $x_1 \equiv x_2$ for some $x_1, x_2 \in X$, then $f(x_1) \equiv f(x_2)$.
\end{lemma}
\begin{proof}
Let $A \subseteq Y$ be an arbitrary open neighbourhood of $f(x_1)$. Since $f$ is continuous, $f^{-1}(A)$ is open. Clearly $x_1 \in f^{-1}(A)$. Since $x_1 \equiv x_2$, we have $x_2 \in f^{-1}(A)$. Then $f(x_2) \in f(f^{-1}(A)) \subseteq A$. Since $A$ was arbitrary, $f(x_1) \equiv f(x_2)$.
\end{proof}

The following theorem says that the quotient map $\eta$ is \emph{universal}.

\begin{theorem}
Let $\eta_X \colon X \to X/{\equiv}$ and $\eta_Y \colon Y \to Y/{\equiv}$ be the quotient maps and $f \colon X \to Y$ an arbitrary continuous map. Then there exists a continuous map $f_{\equiv} \colon X/{\equiv} \to Y/{\equiv}$ such that the diagram below commutes.
\[
\begin{tikzcd}
X \arrow{r}{f} \arrow[swap]{d}{\eta_X} & Y \arrow{d}{\eta_Y} \\%
X/{\equiv} \arrow{r}{f_{\equiv}}& Y/{\equiv}
\end{tikzcd}
\]
\end{theorem}
\begin{proof}
For all equivalence classes $\eta_X(x)$, define $f_{\equiv}(\eta_X(x)) = \eta_Y(f(x))$. This is well-defined: if $x_1 \equiv x_2$, then by the previous lemma we have
\[
f_{\equiv}(\eta_X(x_1)) = \eta_Y(f(x_1)) = \eta_Y(f(x_2)) = f_{\equiv}(\eta_X(x_2)).
\]
By definition, the diagram commutes.

For continuity of $f_{\equiv}$, let $A \subseteq Y/{\equiv}$ be open. Since $\eta_Y$ and $f$ are both continuous, $f^{-1}(\eta^\leftarrow_Y(A))$ is open. Note that
\begin{align*}
\eta_X(x) \in f^{-1}_{\equiv}(A) 	& \iff f_{\equiv}(\eta_X(x)) \in A \\
									& \iff \eta_Y(f(x)) \in A \\
									& \iff f(x) \in \eta^\leftarrow_Y(A) \\
									& \iff x \in f^{-1}(\eta^{\leftarrow}_Y(A)) \\
									& \iff \eta_X(x) \in \eta^\rightarrow_X(f^{-1}(\eta^{\leftarrow}_Y(A))).
\end{align*}
Thus, $f_{\equiv}^{-1}(A) = \eta^\rightarrow_X(f^{-1}(\eta^{\leftarrow}_Y(A)))$, and this is open by corollary \ref{eta is open}. Hence $f_{\equiv}$ is continuous.
\end{proof}

Choosing a representative from each equivalence class gives the following theorem, which states that all topological properties of the Kolmogorov quotient of $X$ hold also in a dense subspace of $X$. If there are infinitely many equivalence classes, then the axiom of choice is required.

\begin{theorem}
The space $X/{\equiv}$ is homeomorphic to a dense subspace of $X$.
\end{theorem}
\begin{proof}
Let $\mu \colon X/{\equiv} \to X$ be a function that picks a representative from each equivalence class. We show that $\mu$ is a homeomorphism between $X/{\equiv}$ and $\mu(X/{\equiv})$. Since equivalence classes are disjoint, $\mu$ is injective. Restricting the codomain to the image $\mu(X/{\equiv})$ makes $\mu$ surjective.

Let $U$ be an arbitrary open subset of $X$. We note that $x \equiv \mu(\eta(x))$, and hence $x \in U$ if and only if $\mu(\eta(x)) \in U$. Consequently,

\begin{align*}
\mu(\eta^\rightarrow(U)) & = \{\mu(\eta(x)) \mid x \in U\} \\
			 & = U \cap \{\mu(\eta(x)) \mid x \in X\} \\
			 & = U \cap \mu(X/{\equiv}),
\end{align*}
that is, the image of every open set of $X/{\equiv}$ is an open set of the subspace $\mu(X/{\equiv})$. Since $\mu$ is bijective, this proves that $\mu^{-1}$ is continuous. Take the inverse image of both sides. By injectivity of $\mu$,
\[
\eta^\rightarrow(U) = \mu^{-1}(U \cap \mu(X/{\equiv})).
\]
Thus the preimage of every open set of $\mu(X/{\equiv})$ is open, and hence $\mu$ is continuous.

We still need to show that $\mu(X/{\equiv})$ is dense in $X$. Let $U$ be a nonempty open set of $X$ and let $x \in U$. Then $\eta(x) \in \eta^\rightarrow(U)$, and hence $\mu(\eta(x)) \in \mu(\eta^\rightarrow(U)) = U \cap \mu(X/{\equiv})$. This shows that the intersection $U \cap \mu(X/{\equiv})$ is nonempty for all nonempty open subsets $U$ of $X$. Then every nonempty open set also intersects $\overline{\mu(X/{\equiv})}$. The complement of $\overline{\mu(X/{\equiv})}$ is open, so it must be empty; therefore $\overline{\mu(X/{\equiv})} = X$.
\end{proof}

The Kolmogorov quotient may have fewer subspaces than the original space. For example, the space $X = \{1,2,3,4\}$ with the clopen basis $\{\{1,2\}, \\ \{3,4\}\}$ has $2^4$ different subspaces, but the quotient $X/{\equiv} = \{\eta(1),\eta(3)\}$ has only $2^2$ different subspaces. The following theorem tells that the quotients of the lost subspaces are still subspaces of $X/{\equiv}$, up to homeomorphism.

\begin{theorem}
Let $X$ be a topological space and $S$ a subspace of $X$. Then the space $S/{\equiv}$ is homeomorphic to some subspace of $X/{\equiv}$.
\end{theorem}
\begin{proof}
Let $\eta \colon X \to X/{\equiv}$ and $\eta_S \colon S \to S/{\equiv}$ be the quotient maps. Let $f \colon S/{\equiv} \to \eta(S)$, $f(\eta_S(x)) = \eta(x)$ for all $\eta_S(x) \in S/{\equiv}$. We show that $f$ is a homeomorphism when $\eta(S)$ is considered as a subspace of $X/{\equiv}$.

First, we note that for all $x, y \in S$,
\begin{align*}
f(\eta_S(x)) = f(\eta_S(y)) & \iff \eta(x) = \eta(y) \\
							& \iff x \in U \text{ iff } y \in U & \text{for all } U \text{ open in } X \\
							& \iff x \in U \cap S \text{ iff } y \in U \cap S & \text{for all } U \text{ open in } X \\
							& \iff \eta_S(x) = \eta_S(y).
\end{align*}
The implication from right to left shows that $f$ is well-defined; the implication from left to right shows that $f$ is injective.

For surjectivity, let $\eta(x) \in \eta(S)$. Then $x \in S$, and hence $\eta_S(x) \in S/{\equiv}$ and $\eta(x) = f(\eta_S(x))$.

The open sets of $S/{\equiv}$ are of the form $\eta_S(U \cap S)$, where $U$ is open in $X$. The open sets of $\eta(S)$ are of the form $\eta(U) \cap \eta(S)$. We note that
\begin{align*}
f(\eta_S(U \cap S)) & = \{f(\eta_S(x)) \mid x \in U \cap S\} \\
					& = \{\eta(x) \mid x \in U \cap S\} \\
					& = \{\eta(x) \mid x \in U\} \cap \{\eta(x) \mid x \in S\} \\
					& = \eta(U) \cap \eta(S).
\end{align*}
Since $f$ is bijective, this proves that $f^{-1}$ is continuous. Take the inverse image of both sides. By injectivity of $f$,
\[
\eta_S(U \cap S) = f^{-1}(\eta(U) \cap \eta(S)).
\]
Thus the preimage of every open set of $\eta(S)$ is open, and hence $f$ is continuous.
\end{proof}

\begin{theorem}
Let $\mathcal{I}$ be a set and $(X_i)_{i \in \mathcal{I}}$ a sequence of topological spaces. The spaces $\left(\prod_{i \in \mathcal{I}} X_i\right)/{\equiv}$ and $\prod_{i \in \mathcal{I}} X_i/{\equiv}$ are homeomorphic.
\end{theorem}
\begin{proof}
Let $\eta$ be the quotient map from $\prod_{i \in \mathcal{I}} X_i$ to $\left(\prod_{i \in \mathcal{I}} X_i\right)/{\equiv}$, and let $\eta_i$ be the quotient map from $X_i$ to $X_i/{\equiv}$ for each $i \in \mathcal{I}$. Define a map $f \colon \left(\prod_{i \in \mathcal{I}} X_i\right)/{\equiv} \to \prod_{i \in \mathcal{I}} X_i/{\equiv}$ from the condition $f(\eta(z))(i) = \eta_i(z(i))$ for all $i \in \mathcal{I}$ and all $z \in \prod_{i \in \mathcal{I}} X_i$. The diagram below should commute. The maps $p_i$ and $\pi_i$ are the canonical projections.
\begin{figure}[!htbp]
\centering
\begin{tikzcd}[column sep = tiny, row sep = small]
 &  & X_i \arrow[drr, "\eta_i"] &  &  \\
\prod_{i \in \mathcal{I}}X_i \arrow[urr, "p_i"] \arrow[dr, "\eta"] &  &  &  & X_i/{\equiv} \\
 & \left(\prod_{i \in \mathcal{I}} X_i \right)/{\equiv} \arrow[rr, "f"] &  & \prod_{i \in \mathcal{I}} X_i/{\equiv} \arrow[ur, "\pi_i"] & 
\end{tikzcd}
\end{figure}

We show that $f$ is a homeomorphism. First, we note that for all $z_1, z_2 \in \prod_{i \in \mathcal{I}} X_i$,
\[
\eta(z_1) = \eta(z_2) \iff \eta_i(z_1(i)) = \eta_i(z_2(i)) \text { for all } i \in \mathcal{I}.
\]
The implication from left to right shows that $f$ is well-defined; the implication from right to left shows that $f$ is injective.

For surjectivity, let $q \in \prod_{i \in \mathcal{I}} X_i/{\equiv}$. For all $i \in \mathcal{I}$, $q(i) = \eta_i(x_i)$ for some $x_i \in X_i$. Let $z \in \prod_{i \in \mathcal{I}} X_i$ be such that $z(i) = x_i$ for all $i \in \mathcal{I}$. Then
\[
q(i) = \eta_i(x_i) = \eta_i(z(i)) = f(\eta(z))(i)
\]
for all $i \in \mathcal{I}$, and hence $q = f(\eta(z))$. We see that $q$ is the image of some $\eta(z) \in \left(\prod_{i \in \mathcal{I}} X_i\right)/{\equiv}$. Hence $f$ is surjective.

The basic open sets of $\prod_{i \in \mathcal{I}} X_i/{\equiv}$ are of the form $\prod_{i \in \mathcal{I}} \eta^\rightarrow_i(U_i)$, where each $U_i$ is open in $X_i$ and $U_i \neq X_i$ for only finitely many $i \in \mathcal{I}$. The basic open sets of $(\prod_{i \in \mathcal{I}} X_i)/{\equiv}$ are of the form $\eta^\rightarrow\left(\prod_{i \in \mathcal{I}} U_i\right)$. We note that
\[
f\left(\eta^\rightarrow\left(\prod_{i \in \mathcal{I}} U_i\right)\right) = \prod_{i \in \mathcal{I}} \eta^\rightarrow_i(U_i).
\]
Since $f$ is bijective, this proves that $f^{-1}$ is continuous. Take the inverse image of both sides. By injectivity of $f$,
\[
\eta^\rightarrow\left(\prod_{i \in \mathcal{I}} U_i\right) = f^{-1}\left(\prod_{i \in \mathcal{I}} \eta^\rightarrow_i(U_i)\right).
\]
Thus the preimage of every basic open set is open, and hence $f$ is continuous.
\end{proof}

\section{Separation and regularity axioms}

The \emph{separation axioms} are properties a topological space can have that guarantee the existence of disjoint neighbourhoods in various situations. The separation axioms are ordered so that $T_i$ implies $T_j$ whenever $i \geq j$. There is also another set of analogous properties called the \emph{regularity axioms} such that $T_i = R_{i-1} \land T_0$. In other words, a space satisfies $T_i$ if and only if it is a Kolmogorov quotient of a space that satisfies $R_{i-1}$. Table \ref{separation and regularity} shows the connection. Some authors require normal and regular spaces to be Hausdorff; we do not.

\begin{table}[!htb]
\centering
\caption{The connection between separation and regularity axioms}
\label{separation and regularity}
\begin{tabular}{|l|l|}
\hline
$X/{\equiv}$           				& $X$            						                            \\ \hline
Kolmogorov ($T_0$)                  & topological space                                                 \\ \hline
Fréchet ($T_1$)                     & symmetric ($R_0$)                                                 \\ \hline
Hausdorff ($T_2$)                   & preregular ($R_1$)                                                \\ \hline
regular Hausdorff ($T_3$)           & regular ($R_2$)                                                   \\ \hline
Tychonoff ($T_{3.5}$)               & completely regular ($R_{2.5}$)                                    \\ \hline
normal Hausdorff ($T_4$)            & normal regular ($R_3$)                                            \\ \hline
completely normal Hausdorff ($T_5$) & completely normal regular ($R_4$)                                 \\ \hline
perfectly normal Hausdorff ($T_6$)  & perfectly normal regular ($R_5$)                                  \\ \hline
\end{tabular}
\end{table}

A topological space $X$ is \emph{symmetric} if for all pairs of topologically distinguishable points $x, y \in X$, there are open sets $U$ and $V$ such that $x \in U$, $y \not \in U$ and $y \in V$, $x \not \in V$.

\begin{theorem} \label{symmetricity condition}
A space $X$ is symmetric if and only if $\eta(x) = \overline{\{x\}}$ for all $x \in X$.
\end{theorem}
\begin{proof}
Suppose first that $\eta(x) = \overline{\{x\}}$ for all $x \in X$. A space with less than two points is always symmetric, so we may assume that there are $x, y \in X$ such that $x \not \equiv y$. Then $x \not \in \eta(y) = \overline{\{y\}}$ and $y \not \in \eta(x) = \overline{\{x\}}$. Equivalently, $x \in \overline{\{y\}}^c$ and $y \in \overline{\{x\}}^c$. We can choose $U = \overline{\{y\}}^c$ and $V = \overline{\{x\}}^c$.

Suppose then that $X$ is symmetric. If $X$ is the empty set, then the claim holds. If all points of $X$ are topologically indistinguishable, then all singleton sets have the same closure, which has to be $X$; hence the claim holds.

Thus we may assume that there are at least two equivalence classes with respect to $\equiv$. By the symmetricity of $X$, for all pairs of points $p, q \in X$, $p \not \equiv q$, there exists an open set $V_{pq}$ such that $q \in V_{pq}$ and $p \not \in V_{pq}$. Let $x \in X$ and
\[
W = \bigcup_{y \in [\eta(x)]^c} V_{xy}.
\]
Since $y \in V_{xy}$ for all $y \in [\eta(x)]^c$, we have $[\eta(x)]^c \subseteq W$. Also, $x \not \in V_{xy}$ for all $y \in [\eta(x)]^c$, so $x \not \in W$. As a union of open sets, $W$ is open. Since $W$ is not a neighbourhood of $x$, it cannot be a neighbourhood of any point of $\eta(x)$; hence $\eta(x) \cap W = \emptyset$. It follows that $W = [\eta(x)]^c$. Since $W$ is open, $\eta(x)$ is closed. Since $x \in \eta(x)$, we get $\eta(x) = \overline{\{x\}}$ from corollary \ref{equivalence class subset of closure}.
\end{proof}

\begin{corollary} \label{singletons are closed in T_1}
A space $X$ is $T_1$ if and only if $\{x\}$ is closed for all $x \in X$.
\end{corollary}
\begin{proof}
Suppose first that $X$ is $T_1$. Then it is both Kolmogorov and symmetric. By Kolmogorovness, $\eta(x) = \{x\}$ for all $x \in X$. By symmetricity, $\eta(x) = \overline{\{x\}}$ for all $x \in X$. Hence $\{x\}$ is closed for all $x \in X$.

Suppose then that $\{x\}$ is closed for all $x \in X$. Then $\eta(x) \subseteq \overline{\{x\}} = \{x\} \subseteq \eta(x)$ for all $x \in X$, and hence $\eta(x) = \overline{\{x\}} = \{x\}$ for all $x \in X$. Thus $X$ is both Kolmogorov and symmetric, and hence $T_1$.
\end{proof}

A topological space $X$ is \emph{preregular} if for all pairs of topologically distinguishable points $x, y \in X$, there are open sets $U$ and $V$ such that $x \in U$, $y \in V$ and $U \cap V = \emptyset$.

\begin{theorem}
If $K_1$ and $K_2$ are disjoint compact subsets of a preregular topological space $X$ and do not have disjoint open neighbourhoods, then there exist $x_1 \in K_1$ and $x_2 \in K_2$ such that $x_1 \equiv x_2$.
\end{theorem}
\begin{proof}
Since $X$ is preregular, $X/{\equiv}$ is Hausdorff. The sets $\eta(K_1)$ and $\eta(K_2)$ are compact by continuity of $\eta$. We show that $\eta(K_1) \cap \eta(K_2) \neq \emptyset$. Suppose to the contrary that $\eta(K_1)$ and $\eta(K_2)$ are disjoint. Disjoint compact subsets of a Hausdorff space have disjoint open neighbourhoods (\cite{Engelking}, p.~124); denote these neighbourhoods by $U_1$ and $U_2$, so that $\eta(K_1) \subseteq U_1$ and $\eta(K_2) \subseteq U_2$. Then $K_1 \subseteq \eta^{-1}(\eta(K_1)) \subseteq \eta^{-1}(U_1)$ and similarly $K_2 \subseteq \eta^{-1}(U_2)$. But $\eta^{-1}(U_1)$ and $\eta^{-1}(U_2)$ are disjoint, and they are open by the continuity of $\eta$. This contradicts the assumption that $K_1$ and $K_2$ do not have disjoint open neighbourhoods. Therefore, $\eta(K_1) \cap \eta(K_2) \neq \emptyset$. Hence there is some $\eta(z) \in \eta(K_1) \cap \eta(K_2)$ such that $\eta(z) = \eta(x) = \eta(y)$ for some $x \in K_1$ and $y \in K_2$.
\end{proof}

\section{Pseudometrics and seminorms} \label{Pseudometrics and seminorms}

In example \ref{L^p}, the Kolmogorov quotient map transformed a seminorm into a norm. In this section we see that this happens for all seminorms.

\begin{definition}
Let $X$ be a set. A \emph{pseudometric on $X$} is a map $d \colon X^2 \to \R$ such that
\begin{enumerate}
\item $d(x,y) \geq 0$ for all $x, y \in X$;
\item $d(x,x) = 0$ for all $x \in X$;
\item $d(x,y) = d(y,x)$ for all $x, y \in X$;
\item $d(x,z) \leq d(x,y) + d(y,z)$ for all $x, y, z \in X$ (\emph{triangle inequality}).
\end{enumerate}
A pseudometric $d$ is a \emph{metric} if $d(x,y) = 0$ implies that $x = y$.
\end{definition}

Let $d$ be a (pseudo)metric on $X$, $x \in X$ and $r$ a positive real number. The set
\[
B(x,r) = \{y \in X \mid d(x,y) < r\}
\]
is the \emph{open ball of radius $r$ centered at $x$}. The set $\{B(x,r) \mid x \in X, r > 0\}$ is a basis for a topology on $X$ (\cite{Munkres}, p.~119). The resulting topological space is called a \emph{(pseudo)metric space} and can be denoted by $(X,d)$. If $x,y \in X$ are such that $d(x,y) = r > 0$, then $x \in B(x,r)$, but $y \not \in B(x,r)$. By lemma \ref{indistinguishability equivalences} (\rm{iii}), $x \equiv y$ if and only if $d(x,y) = 0$.

\begin{example}
Pseudometrics can be used in the context of cellular automata. Given a finite set $A$, let $A^\Z$ denote the set of functions from $\Z$ to $A$. For $x \in A^\Z$, we write $x_j$ for $x(j)$. Also, for $n$, $k \in \Z$, let $[n,k]$ denote the set of integers $m$ such that $n \leq m \leq k$. Finally, for sequences $(a_n)_{n=0}^\infty$ of natural numbers, denote
\[
\limsup_{n \to \infty} a_n = \lim_{n \to \infty} \left( \sup_{m \geq n} a_m \right).
\]
Then
\[
d_B(x,y) = \limsup_{l \to \infty} \frac{\abs{\{j \in [-\el, \el] \mid x_j \neq y_j\}}}{2\el +1}
\]
is the \emph{Besicovitch pseudometric} on $A^\Z$, and
\[
d_W(x,y) = \limsup_{l \to \infty} \max_{k \in \Z} \frac{\abs{\{j \in [k + 1, k + \el] \mid x_j \neq y_j\}}}{\el}
\]
is the \emph{Weyl pseudometric} on $A^\Z$. The topologies induced by these pseudometrics have some advantages to the standard approach, where $A$ is given the discrete topology and $A^\Z$ the product topology; for example, the class of continuous functions from $A^\Z$ to itself is larger (\cite{pseudometrics}).
\end{example}



\begin{theorem} \label{pseudometric to metric}
Let $(X,d)$ be a pseudometric space. Then $d^* \colon (X/{\equiv})^2 \to \R$, $d^*(\eta(x), \eta(y)) = d(x,y)$ for all $x, y \in X$, is a metric on $X/{\equiv}$ that determines the same topology as the quotient map.
\end{theorem}
\begin{proof}
To show that $d^*$ is well-defined, let $x_1, x_2 \in \eta(x)$ and $y_1, y_2 \in \eta(y)$. Using the triangle inequality, we see that
\[
d(x_1, y_1) \leq d(x_1, x_2) + d(x_2, y_1) = d(x_2, y_1) \leq d(x_2, y_2) + d(y_2, y_1) = d(x_2, y_2),
\]
and by a symmetric argument we can prove that $d(x_2, y_2) \leq d(x_1, y_1)$. Thus $d(x_1, y_1) = d(x_2, y_2)$ and $d^*$ is well-defined.

It is easy to see that $d^*$ is a pseudometric on $X/{\equiv}$ by reducing each part of the definition to the corresponding property of $d$. For example, the triangle inequality can be shown as follows: for all $x, y, z \in X$,
\[
d^*(\eta(x),\eta(z)) = d(x,z) \leq d(x,y) + d(y,z) = d^*(\eta(x),\eta(y)) + d^*(\eta(y), \eta(z)).
\]

Since $d(x,y) = 0$ implies $\eta(x) = \eta(y)$ and $d^*(\eta(x), \eta(y)) = d(x,y)$, we see that $d^*$ is a metric. The open balls correspond to those of $X$, so the topology $d^*$ determines is precisely that determined by the quotient map.
\end{proof}

The space $(X/{\equiv}, d^*)$ is called the \emph{metric identification} of $(X,d)$.

We will use the following corollary of the triangle inequality later.

\begin{lemma}[Reverse triangle inequality]
For all $x,y,z \in (X,d)$,
\[
\abs{d(x,z) - d(y,z)} \leq d(x,y).
\]
\end{lemma}
\begin{proof}
(\cite{Searcoid}, p.~3) From the triangle inequality, we get
\[
\begin{cases}
  d(x,z) \leq d(x,y) + d(y,z), \\
  d(y,z) \leq d(y,x) + d(x,z),
\end{cases}
\]
which can be rearranged to
\[
\begin{cases}
  d(x,z) - d(y,z) \leq d(x,y), \\
  d(y,z) - d(x,z) \leq d(x,y).
\end{cases}
\]
The claim follows.
\end{proof}

\begin{definition} \label{seminorm}
Let $K$ be a subfield of $\C$ and $V$ a vector space over $K$. A map $\norm \cdot \colon V \to \R$ is a \emph{seminorm} on $V$ if
\begin{enumerate}
\item $\norm x \geq 0$ for all $x \in V$;
\item $\norm{\lambda x} = \abs \lambda \norm x$ for all $\lambda \in K$ and $x \in V$;
\item $\norm{x + y} \leq \norm x + \norm y$ for all $x, y \in V$.
\end{enumerate}
A seminorm is a \emph{norm} if $\norm x = 0$ implies that $x$ is the zero vector.
\end{definition}

A (semi)norm on $V$ induces a (pseudo)metric on $V$ by defining $d(x,y) = \norm{x-y}$ for all $x,y \in X$. The resulting topological space is called a \emph{(semi-) \\ normed vector space} and can be denoted by $(V, \norm \cdot)$.

\begin{lemma} For all $x, y \in (V,\norm \cdot)$,
\[
\abs{\norm x - \norm y} \leq \norm{x - y}.
\]
\end{lemma}
\begin{proof}
The claim follows from the reverse triangle inequality by substituting $z = 0$.
\end{proof}

The following theorem is analogous to theorem \ref{pseudometric to metric}.

\begin{theorem} \label{seminorm theorem}
Let $(V, \norm \cdot)$ be a seminormed vector space. Then $(V/{\equiv}, \norm \cdot^*)$ is a normed vector space, where
\begin{align*}
\lambda \eta(x) = \eta(\lambda x) 	& \qquad \text{for all } \lambda \in K, x \in V, \\
\eta(x) + \eta(y) = \eta(x + y) 	& \qquad \text{for all } x, y \in V,
\end{align*}
and
\begin{align*}
\norm{\eta(x) }^* = \norm{x} & \qquad \text{for all } x \in V.
\end{align*}
Furthermore, $\norm \cdot^*$ determines the same topology as the quotient map.
\end{theorem}
\begin{proof}
It is straightforward to verify that $V/{\equiv}$ satisfies the axioms of a vector space, for example: for all $\alpha \in K$, $x, y \in V$,
\begin{align*}
		& \alpha[\eta(x) + \eta(y)] \\
= {}	& \alpha \eta(x + y) \\
= {}	& \eta(\alpha(x+y)) \\
= {}	& \eta(\alpha x + \alpha y) \\
= {}	& \eta(\alpha x) + \eta(\alpha y) \\
= {}	& \alpha \eta(x) + \alpha \eta(y).
\end{align*}

Let $d$ be the metric that $\norm \cdot$ induces on $V$. To see that $\norm \cdot^*$ is well-defined, we note that if $x \equiv y$, then $\norm{x-y} = d(x,y) = 0$. Then
\[
0 \leq \abs{\norm x - \norm y} \leq \norm{x - y} = 0,
\]
so $\norm x - \norm y = 0$, that is, $\norm x = \norm y$.

It is easy to see that $\norm \cdot^*$ is a seminorm by reducing each part of the definition to the corresponding property of $\norm \cdot$. For example, for all $x, y \in V$,
\[
\norm{\eta(x) + \eta(y)}^* = \norm{\eta(x+y)}^* = \norm{x + y} \leq \norm x + \norm y = \norm{\eta(x)}^* + \norm{\eta(y)}^*.
\]

Let $\mathbf{0}$ be the zero vector of $V$. If $x \not \equiv \mathbf{0}$, then there is an open ball in $V$ that contains $x$ but not $\mathbf{0}$, or vice versa. In either case, $d(x,\mathbf{0}) > 0$, and
\[
\norm{\eta(x)}^* = \norm x = \norm{x-\mathbf{0}} = d(x,\mathbf{0}) > 0,
\]
which proves that $\norm \cdot ^*$ is a norm. For the metric $d^*$ induced by $\norm \cdot ^*$, we have
\[
d^*(\eta(x), \eta(y)) = \norm{\eta(x) - \eta(y)}^* = \norm{\eta(x - y)}^* = \norm{x - y} = d(x,y),
\]
for all $x, y \in V$. By theorem \ref{pseudometric to metric}, the topology determined by $d^*$ is the same as that determined by the quotient map.
\end{proof}

In section \ref{Uniform spaces}, we will see that pseudometric spaces are completely regular, and consequently their Kolmogorov quotients are Tychonoff (see table \ref{separation and regularity}).

\section{Topological groups}

A \emph{topological group} is a group $G$ with a topology on $G$ that makes multiplication and inversion continuous. More specifically, we want the maps $g_1 \colon G \times G \to G$, $g_1(x,y) = xy$ and $g_2 \colon G \to G$, $g_2(x) = x^{-1}$ to be continuous, when $G \times G$ is given the product topology. Since $g_2$ is bijective and its own inverse, it is a homeomorphism if it is continuous. Our aim is to show that for all $x$, $y \in G$, $x \equiv y$ if and only if $y^{-1}x \in \overline{\{1\}}$, and that $G/{\equiv}$ is a topological group with a naturally arising group operation.

Every group is a topological group, when endowed with the discrete or with the trivial topology. If $\R$ is given the euclidean topology, then $(\R, +)$ and $(\R \setminus \{0\}, \cdot)$ are topological groups.

Throughout this section, $G$ is a topological group. For $A$, $B \subseteq G$ and $a \in G$, we denote
\begin{enumerate}
\item $AB = \{xy \mid x \in A, y \in B\}$;
\item $Aa = A\{a\} = \{xa \mid x \in A\}$;
\item $aA = \{a\}A = \{ax \mid x \in A\}$;
\item $A^{-1} = \{x^{-1} \mid x \in A\}$.
\end{enumerate}

We note that $g_1$ is continuous if and only if
\begin{equation} \label{g_1}
\Forall x, y \in G \colon \Forall W \in \mathcal{N}(xy) \colon \Exists U \in \mathcal{N}(x) \colon \Exists V \in \mathcal{N}(y) \colon UV \subseteq W.
\end{equation}
Similarly, $g_2$ is continuous if and only if
\[
\Forall x \in G \colon \Forall W \in \mathcal{N}(x^{-1}) \colon \Exists U \in \mathcal{N}(x) \colon U^{-1} \subseteq W.
\]

\begin{lemma} \label{g_3}
The maps $g_1$ and $g_2$ are continuous if and only if the map $g_3 \colon G \times G \to G$ defined by $g_3(x,y) = xy^{-1}$ is continuous.
\end{lemma}
\begin{proof}
Suppose first that $g_1$ and $g_2$ are continuous. In (\ref{g_1}), substitute $y \mapsto y^{-1}$. The first $y^{-1}$ can be changed back to $y$, since quantifying over all $y^{-1} \in G$ is the same as quantifying over all $y \in G$. Since $g_2$ is a homeomorphism, $\Exists V \in \mathcal{N}(y^{-1}) \colon UV \subseteq W$ is equivalent to $\Exists V \in \mathcal{N}(y) \colon UV^{-1} \subseteq W$. Hence
\[
\Forall x, y \in G \colon \Forall W \in \mathcal{N}(xy^{-1}) \colon \Exists U \in \mathcal{N}(x) \colon \Exists V \in \mathcal{N}(y) \colon UV^{-1} \subseteq W,
\]
so $g_3$ is continuous.

Suppose then that $g_3$ is continuous. By substituting $x \mapsto 1$ we get the continuity of $g_2$. Then $g_2$ is a homeomorphism, and we can follow the preceding proof in the opposite direction to show the continuity of $g_1$.
\end{proof}

\begin{theorem}
Let $a \in G$. The functions $r_a \colon G \to G$ and $l_a \colon G \to G$ defined by $r_a(x) = xa$ and $l_a(x) = ax$ are homeomorphisms.
\end{theorem}
\begin{proof}
(\cite{Spivak}, p.~12) We prove the claim for $r_a$. The claim for $l_a$ is analogous.

To show surjectivity, let $y \in G$. Since $r_a(ya^{-1}) = ya^{-1}a = y$, $r_a$ is surjective.

To show injectivity, suppose $r_a(x) = r_a(y)$. Then $xa = ya$ and multiplying from the right by $a^{-1}$ gives $x = y$. Hence $r_a$ is injective.

To show continuity, let $W \in \mathcal{N}(xa)$. By continuity of $g_1$, there exists $U \in \mathcal{N}(x)$ and $V \in \mathcal{N}(a)$ such that $UV \subseteq W$. In particular, $r_a(U) = Ua \subseteq W$. Therefore, $r_a$ is continuous. 

Finally, to show that $r_a^{-1}$ is continuous, note that
\[
r_{a^{-1}}(r_a(x)) = xaa^{-1} = x = xa^{-1}a = r_a(r_a^{-1}(x)).
\]
Hence $r_a^{-1} = r_{a^{-1}}$, which is continuous by the preceding argument.
\end{proof}

\begin{corollary} \label{Sa is open/closed}
Let $U \subseteq G$ be open, $F \subseteq G$ be closed, $A \subseteq G$ and $a \in G$. Then
\begin{enumerate}[label = \rm{(\roman*)}]
\item $Ua$ and $aU$ are open;
\item $Fa$ and $aF$ are closed;
\item $U\!A$ and $AU$ are open.
\end{enumerate}
\end{corollary}
\begin{proof}
For claim $\mathrm{(i)}$, we note that $Ua = r_a(U)$ and $aU = l_a(U)$. Claim $\mathrm{(ii)}$ follows similarly. Claim $\mathrm{(iii)}$ follows from $\mathrm{(i)}$, since $U\!A = \bigcup_{x \in A} Ux$ and $AU = \bigcup_{x \in A} xU$.
\end{proof}

\begin{lemma} \label{Kolmogorov group is T_1}
A topological group $G$ is $T_0$ if and only if $G$ is $T_1$.
\end{lemma}
\begin{proof}
(\cite{Raum}, pp.~6--7) Obviously $T_1$ implies $T_0$. Suppose that $G$ is $T_0$. We show first that for every $x \in G \setminus \{1\}$, there is an open set $U_x$ such that $x \in U_x$ and $1 \not \in U_x$. Let $x \in G \setminus \{1\}$. Since $G$ is $T_0$, there is some open set $U$ such that $x \in U$ and $1 \not \in U$ or such that $x \not \in U$ and $1 \in U$. In the first case, let $U_x = U$. In the latter case, let $U_x = U^{-1}x = r_x(g_2(U))$. This is an open set, since $g_2$ and $r_x$ are homeomorphisms. Since $1 \in U$, we have $1 \in U^{-1}$, and hence $x \in U^{-1}x$. Also, since $x \not \in U$, we have $x^{-1} \not \in U^{-1}$, and hence $1 \not \in U^{-1}x$.

Let
\[
W = \bigcup_{x \in G \setminus \{1\}} U_x.
\]
Since $x \in U_x$ for all $x \in G \setminus \{1\}$, we have $G \setminus \{1\} \subseteq W$. Also, $1 \not \in U_x$ for all $x \in G \setminus \{1\}$, so $1 \not \in W$. Hence $W = G \setminus \{1\}$. As a union of open sets, $W$ is open, which proves that $\{1\}$ is closed. Then $r_x(\{1\}) = \{x\}$ is closed for all $x \in G$. By corollary \ref{singletons are closed in T_1}, $G$ is $T_1$.
\end{proof}

\begin{corollary} \label{Kolmogorov quotient group is T_1}
The Kolmogorov quotient $G/{\equiv}$ is $T_1$.
\end{corollary}

\begin{corollary} \label{topological groups are symmetric}
All topological groups are symmetric.
\end{corollary}

A subset $H \subseteq G$ is a \emph{subgroup} of $G$, denoted $H \leq G$, if $HH^{-1} = H$. A subgroup $H$ of $G$ is \emph{normal} or \emph{invariant}, denoted $H \unlhd G$, if $aHa^{-1} = H$ for all $a \in G$. We use the term ``invariant'' to avoid confusion with normal topological spaces. The trivial subgroup $\{1\}$ is always invariant.

\begin{theorem} 
If $H \leq G$, then $\overline{H} \leq G$. If $H \unlhd G$, then $\overline{H} \unlhd G$.
\end{theorem}
\begin{proof}
Recall that if $f \colon X \to Y$ is continuous, then $f(\overline{A}) \subseteq \overline{f(A)}$ for all $A \subseteq X$. If $H \leq G$, then by the continuity of inversion and multiplication,
\[
\overline{H}\,\overline{H}^{-1} \subseteq \overline{H}\,\overline{H^{-1}} \subseteq \overline{HH^{-1}} = \overline{H}.
\]
Hence $\overline{H} \leq G$.

For the latter part of the theorem, we note that $l_a \circ r_{a^{-1}}$ is continuous. If $H \unlhd G$, then
\[
a\overline{H}a^{-1} = (l_a \circ r_{a^{-1}})(\overline{H}) \subseteq \overline{(l_a \circ r_{a^{-1}})(H)} = \overline{aHa^{-1}} = \overline{H}.
\]
Hence $\overline{H} \unlhd G$.
\end{proof}

\begin{corollary}
$\overline{\{1\}} \unlhd G$.
\end{corollary}

Let $H \leq G$. Let $G/H = \{xH \mid x \in G\}$, and let $\phi \colon G \to G/H$, $\phi(x) = xH$ for all $x \in G$. We define a topology on $G/H$ by letting $A \subseteq G/H$ be open if and only if $\phi^{-1}(A)$ is open. We intend to show that this makes $G/H$ a topological group when $H$ is invariant. The reader should recall or verify that the following statements hold for all $H \leq G$ and for all $x, y \in G$:
\begin{enumerate}[label = (\alph*)]
\item The cosets $xH$ partition $G$.
\item $xH = yH$ iff $x \in yH$ iff $y^{-1}x \in H$.
\item If $H \unlhd G$, then $xH = Hx$.
\item If $H \unlhd G$, then the set $G/H$ is a group with the group operation defined by $(xH)(yH) = xyH$.
\end{enumerate}

\begin{lemma} \label{phi is homomorphism}
Let $H \unlhd G$. Then the map $\phi$ is a continuous open group homomorphism.
\end{lemma}
\begin{proof}
Continuity is clear from the definition of the topology. For openness, let $U \subseteq G$ be open. Then
\begin{align*}
\phi^{-1}(\phi(U)) 	& = \{x \in G \mid xH = uH \text{ for some } u \in U\} \\
					& = \{x \in G \mid x \in uH \text{ for some } u \in U\} \\
					& = U\!H,
\end{align*}
which is open by corollary \ref{Sa is open/closed}. Thus $\phi(U)$ is also open.

To show that $\phi$ is a homomorphism, let $x, y \in G$. Then
\[
\phi(xy) = xyH = (xH)(yH) = \phi(x)\phi(y).
\]
\end{proof}

\begin{lemma}
Let $H \unlhd G$. Then $G/H$ is a topological group.
\end{lemma}
\begin{proof}
(\cite{Husain}, p.~59) It suffices to show that the operation $(xH,yH) \mapsto xy^{-1}H$ is continuous. Let $x$, $y \in G$, and let $W \subseteq G/H$ be an open neighbourhood of $xy^{-1}H$. Then $xy^{-1} \in \phi^{-1}(xy^{-1}H) \subseteq \phi^{-1}(W)$ and $\phi^{-1}(W)$ is open in $G$. Since $G$ is a topological group, there are open sets $U$ and $V$ such that $x \in U$, $y \in V$ and $xy^{-1} \in UV^{-1} \subseteq \phi^{-1}(W)$. By lemma \ref{phi is homomorphism},
\[
(xH)(yH)^{-1} \in \phi(U)[\phi(V)]^{-1} = \phi(U)\phi(V^{-1}) = \phi(UV^{-1}) \subseteq W.
\]
Since $U$ and $V$ are open, so are $\phi(U)$ and $\phi(V)$, which proves the claim.
\end{proof}

We can now prove the main result.

\begin{theorem}
$G/{\equiv} = G/\overline{\{1\}}$.
\end{theorem}
\begin{proof}
What we want to show is that for all $x, y \in G$, $x \equiv y$ if and only if $x\overline{\{1\}} = y\overline{\{1\}}$. Let $x, y \in G$ and suppose first that $x \equiv y$. Since $\el_{y^{-1}}$ is continuous, we have $y^{-1}x \equiv 1$ by lemma \ref{continuous functions preserve equivalence classes}. Hence $y^{-1}x \in \eta(1) \subseteq \overline{\{1\}}$, and consequently $x\overline{\{1\}} = y\overline{\{1\}}$.

Suppose now that $x\overline{\{1\}} = y\overline{\{1\}}$. Equivalently, $y^{-1}x \in \overline{\{1\}}$. By corollary \ref{topological groups are symmetric} and theorem \ref{symmetricity condition}, we have $\eta(1) = \overline{\{1\}}$. Then $y^{-1}x \in \eta(1)$, that is, $y^{-1}x \equiv 1$. Now, by the continuity of $\el_y$ and lemma \ref{continuous functions preserve equivalence classes}, $x \equiv y$.
\end{proof}

The result tells us that $G/{\equiv}$ is a topological group with a naturally arising group operation, namely that of $G/\overline{\{1\}}$. Also, $\eta(x) = x\overline{\{1\}} = \overline{\{1\}}x$ for all $x \in G$.

From the proof above, we see that $x \equiv y$ if and only if $y^{-1}x \in \overline{\{1\}}$. Since $\overline{\{1\}}$ is invariant, we can replace the left cosets $x\overline{\{1\}}$ by right cosets $\overline{\{1\}}x$ and use $r_y$ and $r_{y^{-1}}$ in place of $l_y$ and $l_{y^{-1}}$ to show that this is also equivalent to $xy^{-1} \in \overline{\{1\}}$. Of course, we can swap the roles of $x$ and $y$, so $x^{-1}y \in \overline{\{1\}}$ and $yx^{-1} \in \overline{\{1\}}$ are also equivalent to $x \equiv y$.

We have now seen that the equivalence classes are determined by the closure of $\{1\}$. Next we will show that the whole topology is determined by the neighbourhoods of $1$. In the next section, we will use this to show that topological groups are completely regular, strengthening corollary \ref{topological groups are symmetric}.

\begin{lemma} \label{basis lemma}
Let $(X, \tau)$ be a topological space. Suppose that $\mathcal{B}$ is a collection of open sets of $X$ such that for all $U \in \tau$, for all $x \in U$, there is a set $B \in \mathcal{B}$ such that $x \in B \subseteq U$. Then $\mathcal{B}$ is a basis for $\tau$.
\end{lemma}
\begin{proof}
\cite{Munkres}, p.~80.
\end{proof}

\begin{theorem} \label{xN are a basis}
The collection $\mathcal{B} = \{xN \mid N \in \mathcal{N}(1)\}$ is a basis for the topology of $G$.
\end{theorem}
\begin{proof}
Let $U$ be open and $x \in U$. Then $1 \in x^{-1}U$, and $x^{-1}U$ is open, so $x^{-1}U \in \mathcal{N}(1)$. Now $x \in xx^{-1}U = U$, so $x^{-1}U$ satisfies the conditions of lemma \ref{basis lemma}. The claim follows.
\end{proof}

\section{Uniform spaces} \label{Uniform spaces}

Topological spaces originated as a general framework for studying the concept of continuity. Similarly, uniform spaces are a general framework for studying the concept of uniform continuity.

Let $X$ be a set. We denote by $\id_X$ the identity relation on $X$, that is, $\id_X = \{(x,x) \mid x \in X\}$. For all binary relations $U, V  \subseteq X \times X$ and for all $x_0 \in X$, we write $U^{-1} = \{(x,y) \mid (y,x) \in U\}$, $U \circ V = \{(x,z) \mid (x,y) \in U, (y,z) \in V\}$ and $U(x_0) = \{y \mid (x_0,y) \in U\}$.

\begin{definition}
A \emph{uniform structure} or \emph{uniformity} on a set $X$ is a filter $\mathcal{U}$ on $X \times X$ such that the following hold for all $U \in \mathcal U$:
\begin{enumerate}
\item[U1.] $\id_X \subseteq U$.
\item[U2.] $U^{-1} \in \mathcal{U}$.
\item[U3.] There exists $V \in \mathcal U$ such that $V \circ V \subseteq U$.
\end{enumerate}
The relations $U \in \mathcal{U}$ are called \emph{entourages}, and the pair $(X, \mathcal{U})$ is called a \emph{uniform space}.
\end{definition}

We note that if $\mathcal{U}$ is a uniformity, then by U1, $(x,y) \in U \in \mathcal{U}$ implies $(x,y) \in U \circ U$. In other words, $U \subseteq U \circ U$ for all $U \in \mathcal{U}$.

\begin{example}
Every set $X$ has the trivial uniformities $\{\id_X\}$ and $\powerset{X \times X}$.
\end{example}

\begin{example} \label{pseudometric uniformity}
Let $(X,d)$ be a pseudometric space. For all $r > 0$, let
\[
U_r = \{(x,y) \in X \times X \mid d(x,y) < r \}.
\]
Then let $\mathcal{U}_d = \{V \subseteq X \times X \mid U_r \subseteq V \text{ for some } r > 0\}$. It is easy to verify that $\mathcal{U}_d$ is a uniformity on $X$.
\end{example}

\begin{example} \label{topological group uniformity}
Let $G$ be a topological group. For all $N \in \mathcal{N}(1)$ let
\[
L_N = \{(x,y) \in G \times G \mid x^{-1}y \in N\}.
\]
Then let $\mathcal{S}(G) = \{L_N \mid N \in \mathcal{N}(1)\}$. It is easy to verify that $\mathcal{S}(G)$ is a filter. Since $x^{-1}x = 1 \in N$ for all $x \in G$ and for all $N \in \mathcal{N}(1)$, U1 holds for $\mathcal{S}(G)$. The axiom U2 holds because $L_N^{-1} = L_{N^{-1}}$. To see that U3 holds, substitute $x \mapsto 1$ and $y \mapsto 1$ in (\ref{g_1}) to obtain that for all neighbourhoods $W \in \mathcal{N}(1)$, there are neighbourhoods $U,V \in \mathcal{N}(1)$ such that $UV \subseteq W$. Without loss of generality, we may assume that $U$ and $V$ are open. Then $(U \cap V)(U \cap V) \subseteq UV \subseteq W$, and $U \cap V$ is an open neighbourhood of $1$. If $x^{-1}y, y^{-1}z \in U \cap V$, then $x^{-1}z = x^{-1}yy^{-1}z \in W$. Thus
\[
L_{U \cap V} \circ L_{U \cap V} = \{(x,z) \in G \times G \mid x^{-1}y, y^{-1}z \in U \cap V \text{ for some } y \in G\} \subseteq L_W.
\]
\end{example}

In the preceding examples we saw uniformities arise from topologies. Next we show that this process can be inverted to create a topology from a given uniformity.

\begin{theorem} \label{uniform topology}
Let $(X, \mathcal{U})$ be a uniform space. The collection
\[
\tau = \{T \in \powerset X \mid \text{for all } x \in T \text{ there exists } U \in \mathcal{U} \text{ such that } U(x) \subseteq T\}
\]
is a topology on $X$.
\end{theorem}
\begin{proof}
Clearly $\emptyset \in \tau$. We use the fact that $\mathcal{U}$ is a filter to show that the other parts of the definition of topological spaces are satisfied.

To show that $X \in \tau$, we need to show that for all $x \in X$, there exists $U \in \mathcal{U}$ such that $U(x) = X$. By F1, we can take $U = X \times X$.

Let $\mathcal{I}$ be a set and $(T_i)_{i \in \mathcal I}$ a sequence of sets where $T_i \in \tau$ for all $i \in \mathcal{I}$. Then for all $i \in \mathcal{I}$, for all $x \in T_i$, there exists $U_i \in \mathcal{U}$ such that $U_i(x) \subseteq T_i$. Then for all $x \in \bigcup_{i \in \mathcal{I}} T_i$, $U(x) \subseteq \bigcup_{i \in \mathcal{I}} T_i$, where $U = \bigcup_{i \in \mathcal{I}} U_i$. By F2, $U \in \mathcal{U}$; hence $\bigcup_{i \in \mathcal{I}} T_i \in \tau$.

Let $T_1, T_2 \in \tau$. Then for all $x \in T_1 \cap T_2$, there exist $U_1, U_2 \in \mathcal{U}$ such that $U_1(x) \subseteq T_1$ and $U_2(x) \subseteq T_2$. Then for all $x \in T_1 \cap T_2$, $(U_1 \cap U_2)(x) = U_1(x) \cap U_2(x) \subseteq T_1 \cap T_2$. By F3, $U_1 \cap U_2 \in \mathcal{U}$; hence $T_1 \cap T_2 \in \tau$. By induction, the intersection of a finite collection of sets from $\tau$ is a member of $\tau$.
\end{proof}

We call the topology of theorem \ref{uniform topology} the \emph{uniform topology induced by $\mathcal{U}$}. A topological space $(X,\tau)$ is \emph{uniformizable} if there exists a uniformity on $X$ such that the uniform topology it induces is $\tau$. Such a uniformity need not be unique.

\begin{example}
Let $\mathcal{U}_d$ be as in example \ref{pseudometric uniformity}. By lemma \ref{basis lemma}, we can choose $\mathcal{B} = \{U_r(x) \mid r > 0\}$ as a basis for the uniform topology. We note that $U_r(x) = B(x,r)$, so this is exactly the topology determined by $d$.
\end{example}

\begin{example}
Let $\mathcal{S}(G)$ be as in example \ref{topological group uniformity}. By lemma \ref{basis lemma}, we can choose $\mathcal{B} = \{L_N(x) \mid N \in \mathcal{N}(1)\}$ as a basis for the uniform topology. We note that $L_N(x) = xN$, so this is exactly the original topology of $G$ by theorem \ref{xN are a basis}.
\end{example}

We say that a topological space $X$ is \emph{completely regular} or \emph{$R_{2.5}$} if for all closed sets $A \subseteq X$, for all points $x \in X \setminus A$, there is a continuous map $f \colon X \to \R$ such that $f(A) = \{0\}$ and $f(x) = 1$. It turns out that $(X,\tau)$ is uniformizable if and only if it is completely regular (\cite{HAF}, p.442--443). Thus we have shown that pseudometric spaces and topological groups are completely regular. Their Kolmogorov quotients are then Tychonoff (see table \ref{separation and regularity}).

The following theorem is a special case of a theorem on p.~178 of \cite{Kelley}.

\begin{theorem} \label{U(x) is open}
The sets $U(x)$ are open in the uniform topology for all $U \in \mathcal{U}$ and for all $x \in X$.
\end{theorem}
\begin{proof}
Let $U \in \mathcal{U}$ and define $U'(x) = \{y \in X \mid V(y) \subseteq U(x) \text{ for some } V \in \mathcal{U}\}$. We show that $U'(x) = U(x)$; then it is clear that $U(x)$ is open.

To show that $U'(x) \subseteq U(x)$, let $y \in U'(x)$. Then there exists $V \in \mathcal{U}$ such that $V(y) \subseteq U(x)$. By U1, $y \in V(y)$, and hence $y \in U(x)$.

To show that $U(x) \subseteq U'(x)$, let $y \in U(x)$. By U3, there exists $V \in \mathcal{U}$ such that $V \circ V \subseteq U$. Thus $V(y) \subseteq (V \circ V)(y) \subseteq U(x)$. Then $y \in U'(x)$, which proves the claim.
\end{proof}

\begin{theorem}
Let $(X, \mathcal{U})$ be a uniform space equipped with the uniform topology. Then $x \equiv y$ if and only if $(x,y) \in U$ for all $U \in \mathcal{U}$.
\end{theorem}
\begin{proof}
Suppose $x \not \equiv y$. Without loss of generality, we may assume that there is an open set $T$ such that $x \in T$ and $y \not \in T$. Since $T$ is open and $x \in T$, there exists $U \in \mathcal{U}$ such that $U(x) \subseteq T$. Since $y \not \in T$, we have $y \not \in U(x)$, or equivalently, $(x,y) \not \in U$.

Suppose there is an entourage $U \in \mathcal{U}$ such that $(x,y) \not \in U$. Then $y \not \in U(x)$. By theorem \ref{U(x) is open}, $U(x)$ is open, and by U1, $x \in U(x)$. Then $x \not \equiv y$.
\end{proof}

By U2, the pair $(x,y)$ is in every entourage if and only if the pair $(y,x)$ is in every entourage.

\begin{corollary}
On a uniform space $(X, \mathcal{U})$ equipped with the uniform topology, the topological indistinguishability relation is the intersection of all entourages:
\[
{\equiv} = \bigcap_{U \in \mathcal{U}}U.
\]
In particular, for all $x \in X$,
\[
\eta(x) = \bigcap_{U \in \mathcal{U}} U(x).
\]
\end{corollary}

\section{A homotopy equivalence of Alexandrov- \\ discrete spaces}

In spaces where all intersections of open sets are open, the Kolmogorov quotient map is a homotopy equivalence, as defined below. We present the proof of this result from \cite{McCord}. The proof uses the axiom of choice.

A topological space is \emph{Alexandrov-discrete} if all intersections of open sets are open. All finite spaces are Alexandrov-discrete, as is the space of natural numbers with a basis consisting of the sets $V_n = \{m \in \N \mid m \geq n\}$.

We denote
\[
U_x = \smashoperator[r]{\bigcap_{\substack{U \text{ open} \\ x \in U}}} \enspace U = \smashoperator[r]{\bigcap_{U \in \mathcal N(x)}} \enspace U,
\]
and call the set $U_x$ the \emph{hull of $x$}. By the same proof as lemma \ref{indistinguishability equivalences} (vii), except that we replace closed sets with open sets, we see that $x \equiv y$ if and only if $U_x = U_y$. In an Alexandrov-discrete space, each $U_x$ is open, and the collection $\{U_x \mid x \in X\}$ is a basis for the topology.

We also define a relation $\leq$ on an Alexandrov-discrete space by letting $x \leq y$ if and only if $x \in U_y$. Note that this is also equivalent to $U_x \subseteq U_y$. This relation is a \emph{preorder}, that is, it is reflexive and transitive. It is a partial order if and only if the space is $T_0$. We say a map $f \colon X \to Y$ between Alexandrov-discrete spaces is \emph{order-preserving} if $x_1 \leq x_2$ implies $f(x_1) \leq f(x_2)$ for all $x_1, x_2 \in X$.

\begin{lemma}
Let $X$ and $Y$ be Alexandrov-discrete spaces and $f \colon X \to Y$ a map. Then $f$ is order-preserving if and only if $f$ is continuous.
\end{lemma}
\begin{proof}
Suppose $f$ is order-preserving. Let $x \in X$. For all $x' \in U_x$, we have $f(x') \in U_{f(x)}$, since $f$ preserves the order. Hence $f(U_x) \subseteq U_{f(x)}$. For all neighbourhoods $V$ of $f(x)$, the set $U_x$ is a neighbourhood of $x$ such that $f(U_x) \subseteq U_{f(x)} \subseteq V$; that is, $f$ is continuous at point $x$. Since $x$ was arbitrary, $f$ is continuous.

Suppose then that $f$ is not order-preserving. Then there are points $x, x' \in X$ such that $x' \in U_x$ but $f(x') \not \in U_{f(x)}$. Hence $f(U_x) \not \subseteq U_{f(x)}$. Since $U_x$ is the smallest neighbourhood of $x$, for all neighbourhoods $U$ of $x$ we have $f(U) \not \subseteq U_{f(x)}$. Hence $f$ is not continuous at point $x$, and therefore not continuous.
\end{proof}

\begin{lemma}
If $(X, \tau)$ is Alexandrov-discrete, then $\eta^\rightarrow(U_x) = U_{\eta(x)}$.
\end{lemma}
\begin{proof}
Denote the topology of $X/{\equiv}$ by $\tau_{\equiv}$. Under the quotient map $\eta$, every open neighbourhood $V$ of $\eta(x)$ has an open preimage that is a neighbourhood of $x$. Conversely, every open neighbourhood $U$ of $x$ maps to an open neighbourhood of $\eta(x)$ by corollary \ref{eta is open}. Consequently,
\[
\bigcap_{\substack{U \in \tau \\ x \in U}} \eta^\rightarrow(U) = \smashoperator[r]{\bigcap_{\substack{V \in \tau_{\equiv} \\ \eta(x) \in V}}} \enspace V = U_{\eta(x)}.
\]
Now we have
\[
\eta^\rightarrow(U_x) = \eta^\rightarrow\left(\bigcap_{\substack{U \in \tau \\ x \in U}} U\right) \subseteq \bigcap_{\substack{U \in \tau \\ x \in U}} \eta^\rightarrow(U) = U_{\eta(x)}.
\]
Since $\eta^\rightarrow(U_x)$ is an open neighbourhood of $\eta(x)$, we have $U_{\eta(x)} \subseteq \eta^\rightarrow(U_x)$. Therefore, $\eta^\rightarrow(U_x) = U_{\eta(x)}$.
\end{proof}

\begin{corollary} \label{eta preserves order}
If $X$ is Alexandrov-discrete, then for all $x,y \in X$, $\eta(x) \leq \eta(y)$ if and only if $x \leq y$.
\end{corollary}
\begin{proof}
By the previous lemma, $x \in U_y$ implies $\eta(x) \in \eta^\rightarrow(U_y) = U_{\eta(y)}$. Conversely, $\eta(x) = U_{\eta(y)}$ implies $x \in \eta^\leftarrow(\eta^\rightarrow(U_y)) = U_y$, since $U_y$ is a Borel set.
\end{proof}

\begin{definition}
Let $X$ and $Y$ be topological spaces with $f \colon X \to Y$ and $g \colon X \to Y$ continuous functions. Let $I$ be the unit interval $[0,1]$ taken as a subspace of $\R$. The function $f$ is \emph{homotopic to} $g$ if there exists a continuous map $F \colon X \times I \to Y$ such that $F(x,0) = f(x)$ and $F(x,1) = g(x)$ for all $x \in X$.
\end{definition}

Intuitively, the second argument of $F$ can be interpreted as time; then $F$ describes the function $f$ turning into the function $g$ in a continuous manner over time. Clearly every continuous function is homotopic to itself; just let $F(x,t) = f(x)$ for all $x \in X$, $t \in I$.


\begin{definition}
Topological spaces $X$ and $Y$ are \emph{homotopy equivalent} if there exist continuous maps $f \colon X \to Y$ and $g \colon Y \to X$ such that $g \circ f$ is homotopic to $\id_X$, and $f \circ g$ is homotopic to $\id_Y$. In this case, we say that $f$ and $g$ are \emph{homotopy equivalences}.
\end{definition}

\begin{theorem}
If $X$ is an Alexandrov-discrete space, then the quotient map $\eta$ is a homotopy equivalence.
\end{theorem}
\begin{proof}
Let $\mu \colon X/{\equiv} \to X$ be a map that picks a representative from each equivalence class. Then $\eta \circ \mu = \id_{X/{\equiv}}$. By corollary \ref{eta preserves order}, $\mu$ is order-preserving and hence continuous.

We need to show that $\pi \colon X \to X$, $\pi = \mu \circ \eta$ is homotopic to $\id_X$. For all $x \in X$, $\eta(\pi(x)) = \eta(\mu(\eta(x))) = \eta(x)$. Hence 
\begin{equation} \label{hull of pi(x) is hull of x}
U_{\pi(x)} = U_x.
\end{equation}
Let $F \colon X \times I \to X$ be defined by
\[
F(x,t) =
\begin{cases}
x		& \text{if } 0 \leq t < 1, \\
\pi(x) 	& \text{if } t = 1.
\end{cases}
\]
To show that $F$ is continuous, let $(x,s) \in X \times I$. Now $U_x \times I$ is a neighbourhood of $(x,s)$. Let $(y,t) \in U_x \times I$. If $0 \leq t < 1$, then $F(y,t) = y \in U_x$. On the other hand, if $t = 1$, then $F(y,t) = \pi(y) \in U_{\pi(y)} = U_y \subseteq U_x$. By equation (\ref{hull of pi(x) is hull of x}), $U_{F(x,s)} = U_x$. Therefore, $F(U_x \times I) = U_x = U_{F(x,s)} \subseteq V$ for all $V \in \mathcal{N}(F(x,s))$. Thus $F$ is continuous at point $(x,s)$. Since $(x,s)$ was arbitrary, $F$ is continuous.
\end{proof}

\end{document}